\documentclass[oneside]{amsart}
\usepackage[colorlinks=true,allcolors=blue!80!black]{hyperref}
\usepackage{tikz-cd}
\usetikzlibrary{matrix, snakes, patterns, shadows.blur, backgrounds}

\usepackage{amssymb, amsthm}
\usepackage{enumerate, amsfonts, amsxtra,amsmath,latexsym,epsfig, color}
\usepackage{mathrsfs, MnSymbol}
\usepackage{geometry}                		
\usepackage{graphicx}
\usepackage{subcaption}
\usepackage{autobreak}
\usepackage{color}
\usepackage{amsmath,amscd}	
\usepackage{mathtools}
\usepackage{stmaryrd}
\usepackage{animate}
\usepackage{tikz-cd}
\usepackage{soul}
\usepackage{faktor}
\usepackage{comment}

\usepackage{textgreek}  

\input xy
\xyoption{all}

\newtheorem {theorem}{Theorem}[section]
\newtheorem {lemma} [theorem] {Lemma}
\newtheorem {proposition} [theorem] {Proposition}
\newtheorem {corollary} [theorem] {Corollary}

\theoremstyle{definition}
\newtheorem{definition}[theorem]{Definition}
\newtheorem{remark}[theorem]{Remark}

\DeclareMathOperator{\lcm}{lcm}

\DeclareMathOperator{\Ind}{Ind}

\title{Exceptional Primes in Notions of Arithmetic Similarity}
\author{Shaver Phagan}
\address{Department of Mathematics\\ Purdue University, West Lafayette, IN}
\email{phagan@purdue.edu}

\begin{document}
\maketitle
\section*{Abstract}
 A notion of arithmetic similarity between number fields is defined by requiring equality of some arithmetic statistics over all but finitely many rational primes.  The exceptional set is empty in all previously studied cases, but existing methods for proving this are eclectic. We give a group theoretic explanation for those cases and show that some notions of arithmetic similarity admit a non-empty exceptional set.  Furthermore, we prove a formula for actions of finite cyclic groups on finite sets, which augments a classical result of Gassmann and might also be of independent interest.

\section{Introduction}
We must begin by setting terminology for the discussion to follow. Denote by \(\mathcal{M}_\mathbb{N}^f\) the set of \(\mathbb{N}-\)valued multisets with finite support and identify \(\mathbb{N}\subset\mathcal{M}_\mathbb{N}^f\) by the rule $$\mathbb{N}\ni n\mapsto\{n\}=\left(k\mapsto\begin{cases}
1, & k=n\\
0, & k\neq n
\end{cases}\right)\in\mathcal{M}_\mathbb{N}^f.$$
The \textbf{elements} of a multiset \(s\) are the elements in its support \(\text{supp}(s)\).
\begin{definition}
 The \textbf{splitting type} \(S_F(p)\) of a rational prime \(p\) in a number field \(F\)  is defined as the multiset of residual degrees of \(F\) over \(p\).  Precisely, if \(p\mathcal{O}_F=\mathfrak{p}_1^{e_1}\cdots\mathfrak{p}_n^{e_n}\) is the decomposition of \(p\) into prime ideals in the ring of integers of \(F\), then \(S_F(p)\) is the \textit{multiset} \(\{f_i\}_{i=1}^n\), where \(f_i=[\mathcal{O}_F/\mathfrak{p}_i:\mathbb{Z}/p\mathbb{Z}]\).
\end{definition}
\begin{definition}
Let \(\mathbb{P}\) be the set of rational primes and \(\mathcal{N}\) the set of number fields.  A \textbf{splitting statistic} is a function \(s:\mathbb{P}\times\mathcal{N}\rightarrow \mathcal{M}_\mathbb{N}^f\) of the form \(s(p,F)(n)=f_n(S_F(p))\) for some functions \(f_n:\mathcal{M}_\mathbb{N}^f\rightarrow\mathbb{N}\) and such that \(\text{supp}(s(\cdot,F))\) and \(|s(\cdot,F)|\) are uniformly bounded.
\end{definition}

The splitting statistics we consider in this note are always either \((p,F)\mapsto S_F(p)\) or a straightforward \(\mathbb{N}\)-valued function on \(S_F(p)\) (c.f. \ref{q}-\ref{uc} below).

\begin{definition}
Given a splitting statistic \(s\), a \textbf{notion of arithmetic similarity} induced by \(s\) is an equivalence relation \(\sim\) on \(\mathcal{N}\), given by the rule \(K\sim F\) if \(s(p,F)=s(p,K)\) for all but finitely many \(p\in\mathbb{P}\), or \textit{cofinitely} many \(p\).  The set of primes where \(s(p,F)\neq s(p,K)\) will be referred to as the \textbf{exceptional set} of \(s\) for \(K\) and \(F\), and we will say that a notion of arithmetic similarity \textit{admits a non-empty exceptional set} if it is induced by a splitting statistic with a non-empty exceptional set for some number fields.
\end{definition}
A notion might be defined with an equality almost everywhere, in the sense of the natural density on \(\mathbb{P}\).  Since our notions are required to factor through the splitting type, a Chebotarev argument shows that this implies equality at cofinitely many primes.  In particular, the exceptional set is always a subset of the ramified primes of a Galois closure.\\

Notions of arithmetic similarity have received significant attention (c.f. \cite{perlis_equation_1977, klingen_zahlkorper_1978, jehne_kronecker_1977, lochter_new_1995, lochter_weakly_1994, lochter_weakly_1994-1, stuart_new_1995, klingen_arithmetical_1998, bogart_ell_2018}) and are a source of interesting partitions of \(\mathcal{N}\), which reflect various properties of Dedekind zeta functions.  For instance, if the underlying splitting statistic is \(s(p,F)=S_F(p)\), then \(F\sim K\) if and only if \(\zeta_F=\zeta_K\), and if \(s(p,F)=\min S_F(p)\), then \(F\sim K\) if and only if \(N_F=N_K\), where \(N_F\subset\mathbb{N}\) is the set of positive integers \(n\) such that \(n^{-s}\) has non-zero coefficient in the expression of \(\zeta_F\) as a Dirichlet series.  A recurring theme of this niche has been emptiness of the exceptional set of a splitting statistic inducing a given notion.  Indeed, emptiness has been shown a guarantee for every case in which the issue has been addressed.  We are interested in whether or not this phenomenon holds across all notions of arithmetic similarity.  It turns out the answer is no (c.f. Theorem \ref{main} and Section \ref{nempty}).  We are also interested in developing a common perspective from which emptiness may be proven in the known cases.\\\\ To the author's knowledge, the following is an exhaustive list of notions of arithmetic similarity that have appeared in the literature.  Number fields \(F\) and \(K\) are

\begin{enumerate} 
	\item\label{q} \textbf{arithmetically equivalent} if \(S_F(p)=S_K(p)\) for cofinitely many \(p\)
	\item\label{k} \textbf{Kronecker equivalent} if \(1\in S_F(p)\) if and only if \(1\in S_K(p)\) for cofinitely many \(p\)
	\item\label{wk} \textbf{weakly Kronecker equivalent} if \(\gcd(S_{F}(p))=\gcd(S_{K}(p))\) for cofinitely many \(p\)
	\item\label{uc} \textbf{ultra-coarsely arithmetically equivalent} if \(\lcm(S_{F}(p))=\lcm(S_{K}(p))\) for cofinitely many \(p\).
\end{enumerate}
\begin{remark}
Weak Kronecker equivalence has also been called \textit{local gcd equivalence} in \cite{linowitz_locally_2017}.
\end{remark}
Emptiness of the exceptional set was shown for \ref{q} using a functional equation and analytic properties of Dedekind zeta functions \cite{perlis_equation_1977}.  Emptiness was established for \ref{k} by explicit calculation of a certain linear representation \cite{lochter_new_1995}. For \ref{wk}, emptiness followed from a detailed analysis of specialized functions \cite{lochter_weakly_1994}. In \cite{phagan_corresponding_2024}, we introduced notion \ref{uc} and did not address the emptiness issue.  We give new proofs (from first principles) of emptiness of the exceptional set for \ref{q}-\ref{wk} from a unified perspective using only elementary properties of group actions.  This is desirable, as the splitting type admits a natural description through orbits of group actions (c.f. \cite{perlis_equation_1977, serre_local_1979} or proof of Theorem \ref{qeq} in this note).  We also prove 
\begin{theorem}\label{main}
There exist notions of arithmetic similarity admitting a non-empty exceptional set.
\end{theorem}
The next paragraph summarizes the structure of the present note.\\\\  In Section \ref{group}, we record some relevant group theory and prove Theorem \ref{formula}, which could be of independent interest. In Section \ref{galois}, we show that equality of splitting statistics at cofinitely many primes forces certain relations between Galois groups.  This frames characterizations of the splitting type that are particularly amenable to group-theoretic probing--even at ramified primes--and which are used to establish emptiness of the exceptional sets for the statistics inducing notions \ref{q}-\ref{wk} in Section \ref{empty}.  The technical cornerstone of our approach lies in Theorem \ref{formula} and Lemma \ref{charactersandresidues}.  Indeed, the former adapts a technique due to Gassmann to account for the possibility of ramification, and the latter is an observation tailor-made for studying the arithmetic implications of conjugacy covering relations between (subgroups of) Galois groups.  In Section \ref{nempty}, we show that the exceptional set can be non-empty for notion \ref{uc} and another notion, and we conclude with some parting thoughts on a general study of notions of arithmetic similarity in Section \ref{parting}.

\subsection*{Acknowledgements}
I would like to thank my advisor, Ben McReynolds, for all his help, and for suggesting the line of inquiry which ultimately led to this project.  I would also like to thank Andy Booker for helpful feedback on an earlier draft.

\section{Some Group Theory}\label{group}
In what follows, \(G,G_1,G_2\) are finite groups. For a subset \(A\subset G\), we will use the notation \(A^m\) for the set \(\{a^m:\text{ }a\in A\}\).  When \(A=C\) is a cyclic group and \(m\) divides \(|C|\), note that \(C^m\) is the subgroup of \(C\) of index \(m\).
\begin{definition}
The subgroups \(G_1\) and \(G_2\) are \textbf{almost conjugate} in \(G\) if \(|g^G\cap G_1|=|g^G\cap G_2|\) for all \(g\in G\), where \(g^G\) is the conjugacy class of \(g\) in \(G\).
\end{definition}
\begin{definition}
The subgroups \(G_1\) and \(G_2\) \textbf{conjugacy cover} one another in \(G\) if $$
\bigcup\limits_{g\in G}gG_1g^{-1}=\bigcup\limits_{g\in G}gG_2g^{-1}.
$$
\end{definition}
The following lemma is well-known, but we give a proof for the sake of completeness.  
\begin{lemma}\label{SylowPullBack}
If \(\pi:G\rightarrow H\) is an epimorphism of finite groups, then a subgroup of \(H\) is Sylow \(p\) if and only if it is of the form \(\pi(G_p)\), where \(G_p\) is a Sylow \(p\)-subgroup of \(G\).
\end{lemma}
\begin{proof}
Let \(G_p\subset G\) be a Sylow \(p\)-subgroup, \(H'=\pi(G_p)\), and \(N=\ker\pi\).  There is an isomorphism of \(N\)-sets 
$$
NG_p/G_p\simeq N/N\cap G_p,
$$
and clearly \([NG_p:G_p]\) is not divisible by \(p\), so \(N\cap G_p\subset N\) is a Sylow \(p\)-subgroup of \(N\), which we will denote \(N_p\).   Since there is an isomorphism \(H'\simeq G_p/N_p\), we conclude that \(H'\) is Sylow \(p\) in \(H\), by order considerations.  Now, let \(H_p\) be a Sylow \(p\)-subgroup of \(H\), and let \(G'=\pi^{-1}(H_p)\). If \(G'_p\) is a Sylow \(p\)-subgroup of \(G'\), then \(N_p=N\cap G'_p\) is a Sylow \(p\)-subgroup of \(N\), by our previous reasoning.  But then 
$$
G'_p/N_p\simeq H_p,
$$
so \(G'_p\) is in fact a Sylow \(p\)-subgroup of \(G\), by order considerations.  
\end{proof}

We conclude this section with a formula encoding the combinatorics of a large class of group actions, which might be of independent interest.  Let \(C\) be a finite cyclic group of order \(n\), and let \(S\) be a finite (left) \(C\)-set. Given a divisor \(m\) of \(n\), let \(M_m=|C^m\backslash S|\), and for integers \(t,k\), define 
$$
f_k(t)=\prod\limits_{p|\gcd(k,t)\text{ prime}}p^{v_p(t)},
$$
where \(v_p\) is the \(p\)-adic valuation, so that
$$
v_p(f_k(t))=\begin{cases}
v_p(t), & p|k\text{ and }p|t\\
0, & \text{otherwise}
\end{cases}
$$
\begin{lemma}\label{tr}
If \(T|m|n\), and \(\frac{m}{T}\) is square-free, and \(\Omega\) is the prime omega function counting prime divisors of a positive integer with multiplicity, then
\begin{equation}\label{treq}
\frac{\sum\limits_{T|d|m}(-1)^{\Omega(md)}M_{d}}{\prod\limits_{q|\frac{m}{T} \text{ prime}}(q-1)}=\sum\limits_{\frac{m}{T}f_{m/T}(T)|d|n}\gcd(T,d)a_{d},
\end{equation}
where \(a_d\) is the number of elements with fiber of cardinality \(d\) in the natural projection \(S\rightarrow C\backslash S\), i.e. \(a_d\) is the number of \(C\)-orbits in \(S\) with cardinality \(d\).
\end{lemma}
\begin{proof}
We induct on \(\Omega(m/T)\).  If \(\Omega(m/T)=0\), Equation \ref{treq} reduces to \(M_m=\sum\limits_{d|n}\gcd(m,d)a_{d}\).  By definition \(M_n=|S|=\sum\limits_{d|n}da_{d}\), and if \(\omega\in S\) has \(|C\omega|=d\), then by orbit-stabilizer \(\text{Stab}_C(\omega)=C^{d}\).  For such \(\omega\), we know that \(|C^{m}\omega|=[C^m:C^{d}\cap C^m]=\frac{\lcm(m,d)}{m},\) so 
$$
M_{m}=\sum\limits_{d|n}\frac{md}{\lcm(m,d)}a_{d}=\sum\limits_{d|n}\gcd(m,d)a_{d},
$$
and the desired formula holds.  Suppose that \(\Omega(m/T)>0\), and observe that, since \(m/T\) is square-free, if \(p|\frac{m}{T}\) we have 
$$
\sum\limits_{T|d|m}(-1)^{\Omega(md)}M_{d}=\sum\limits_{pT|d|m}(-1)^{\Omega(md)}M_{d}+\sum\limits_{T|d|\frac{m}{p}}(-1)^{\Omega(md)}M_{d},
$$
so by induction
$$
\frac{\sum\limits_{T|d|m}(-1)^{\Omega(md)}M_{d}}{\prod\limits_{q|\frac{m}{pT} \text{ prime}}(q-1)}=\sum\limits_{\frac{m}{pT}f_{m/pT}(pT)|d|n}\gcd(pT,d)a_{d}-\sum\limits_{\frac{m}{pT}f_{m/pT}(T)|d|n}\gcd(T,d)a_{d}
$$
$$
=\sum\limits_{\frac{m}{pT}f_{m/pT}(T)|d|n}(\gcd(pT,d)-\gcd(T,d))a_{d}
$$
$$
=(p-1)\sum\limits_{\frac{m}{T}f_{m/T}(T)|d|n}\gcd(T,d)a_{d}.
$$
The second equation holds because \(f_{m/pT}(pT)=f_{m/pT}(T)\), and the last equation holds because \(\gcd(pT,d)=\gcd(T,d)\) if and only if \(v_p(d)\leq v_p(T)\), while \(\gcd(pT,d)=p\gcd(T,d)\) otherwise.
\end{proof}
Given \(T,m,n\) as in the lemma and such that \(\frac{m}{T}\) is maximal square-free (note this uniquely determines \(T\)), set 
$$
N_m=\frac{1}{T}\frac{\sum\limits_{T|d|m}(-1)^{\Omega(md)}M_{d}}{\prod\limits_{q|\frac{m}{T} \text{ prime}}(q-1)},
$$
and observe that
\begin{equation}\label{max}
N_m=\sum_{m|d|n}a_{d}.
\end{equation}

Indeed, any prime factor of \(T\) divides \(\frac{m}{T}\), by choice of \(T\).  Therefore, \(f_{m/T}(T)=T\).  Furthermore, since in this case \(m\) divides \(d\) in the right hand side of Equation \ref{treq}, so does \(T\), so \(\gcd(T,d)=T\).  Hence, Equation \ref{max} follows from Lemma \ref{tr}.

\begin{theorem}\label{formula}
If \(m\) is a divisor of \(n\), then
$$
a_m=\sum\limits_{d|\frac{n}{m}}\mu(d)N_{md},
$$
where \(\mu\) is the M{\"o}bius function.
\end{theorem}
\begin{proof}
Let \(f(m)=N_{\frac{n}{m}}\) and \(g(m)=a_{\frac{n}{m}}\).  Equation \ref{max} can be rewritten
$$
f(m)=\sum\limits_{d|m}g(d).
$$
Applying the M{\"o}bius inversion formula proves the claim.
\end{proof}
Theorems \ref{formula} and \ref{qeq} taken together complement the argument, originally due to Gassmann, around Lemma 1 (the Gassmann Lemma) of \cite{perlis_equation_1977}.  Indeed, the latter famously shows that almost conjugacy is equivalent to an equality of splitting types at all unramified primes, but the reasoning there breaks down when there is ramification, so the functional equation must be used to prove emptiness of the exceptional set.  Our results allow one to deduce equality of splitting types at both unramified and ramified primes as a combinatorial consequence of an almost conjugacy assumption.  In fact, Theorem \ref{formula} may be used in tandem with the Gassmann Lemma to recover the following characterization of almost conjugacy (for others, see \cite{sutherland_stronger_nodate} Proposition 2.6), which is essentially the main theorem of \cite{stuart_new_1995}.
\begin{corollary}
Subgroups \(G_1\) and \(G_2\) of \(G\) are almost conjugate if and only if for each cyclic subgroup \(C\) of \(G\), there is an equality of cardinalities \(|C\backslash G/G_1|=|C\backslash G/G_2|\).
\end{corollary}
\begin{proof}
That almost conjugacy implies the cardinality equality follows directly from the Gassmann Lemma.  If \(|C\backslash G/G_1|=|C\backslash G/G_2|\) for every cyclic subgroup \(C\) of \(G\), then for a specific cyclic group \(C\subset G\), we know that \(|C^m\backslash G/G_1|=|C^m\backslash G/G_2|\) for every integer \(m\), so the coset types of \(C\backslash G/G_1\) and \(C\backslash G/G_2\) are the same, by Theorem \ref{formula}.  Almost conjugacy then follows from the Gassmann Lemma.
\end{proof}

\section{Some Galois Theory}\label{galois}
We fix some notation before proceeding.  Given a place \(\nu\) of \(F\), we denote the residue field at \(\nu\) by \(\kappa(\nu)\), and if \(E/F\) is a finite extension (not necessarily Galois) with \(\omega\) a place of \(E\) over \(\nu\), we denote the group of the cyclic extension \(\kappa(\omega)/\kappa(\nu)\) by \(\mathfrak{g}_{\omega/\nu}\). Recall that \(f\) is a residual degree of \(E\) over \(\nu\) if and only if there is \(\omega\) over \(\nu\) such that \(|\mathfrak{g}_{\omega/\nu}|=f\). If \(E/F\) is Galois with group \(G\), we identify \(\text{Gal}(E_\omega/F_\nu)\) with the decomposition group \(\mathcal{D}_{\omega/\nu}\subset G\). Finally, recall that the inertia subgroup \(\mathcal{I}_{\omega/\nu}\) is defined as the kernel of the canonical epimorphism \(\mathcal{D}_{\omega/\nu}\rightarrow\mathfrak{g}_{\omega/\nu}\), and if \(F'\) is a subfield of \(E\) containing \(F\), with a place \(\eta\) divisible by \(\omega\), then \(\mathcal{D}_{\omega/\eta}=\mathcal{D}_{\omega/\nu}\cap G_{F'}\) and \(\mathcal{I}_{\omega/\eta}=\mathcal{I}_{\omega/\nu}\cap G_{F'}\) (c.f. Proposition 22.a in Chapter I of \cite{serre_local_1979}). \\\\
To simplify notation, given a Galois extension \(K/\mathbb{Q}\) with group \(G\) and a place \(\omega\) of \(K\) over \(p\), we will write \(\mathcal{D}, \mathcal{I}, \mathfrak{g}\), respectively, instead of \(\mathcal{D}_{\omega/p}, \mathcal{I}_{\omega/p}, \mathfrak{g}_{\omega/p}\), so long as it won't lead to confusion. Furthermore, we arrange that \(\mathcal{D}=\mathcal{I}C\), where \(C\) is the cyclic group generated by some preimage \(c\) of a generator of \(\mathfrak{g}\) under the projection \(\mathcal{D}\rightarrow\mathfrak{g}\).  Observe that if \(F\) is a subfield of \(K\), and \(\nu\) is the place of \(F\) divisible by \(\omega\), then \(\mathcal{I}\mathcal{D}_{\omega/\nu}\) is the kernel of the composition \(\mathcal{D}\rightarrow\mathfrak{g}\rightarrow\mathfrak{g}_{\nu/p}\), so that there is \(g\in G\) such that \(g\mathcal{I}c^mg^{-1}\cap G_F\neq\emptyset\) if and only if there is \(f\in S_F(p)\) such that \(f\) divides \(m\).  We record this as a lemma, highlighting the generic case that \(p\) is unramified in \(K\).  Recall that \(\Ind_H^G(1_H)(g)\neq 0\) if and only if \(g^G\cap H\neq\emptyset\).

\begin{lemma}\label{charactersandresidues}
With notation as above, the kernel of the natural map \(\mathcal{D}\rightarrow\mathfrak{g}_{\nu/p}\) is \(\mathcal{I}\mathcal{D}_{\omega/\nu}\), and if \(m\) is a positive integer, there is \(g\in G\) such that $$g\mathcal{I}c^mg^{-1}\cap G_F\neq\emptyset$$ if and only if there is \(f\in S_F(p)\) such that \(f|m\).  In particular, if \(p\) is unramified in \(K\), then \(\chi_F(F_p^m)\neq 0\) if and only if there is \(f\in S_F(p)\) such that \(f|m\), where \(\chi_F=\Ind_{G_{F}}^G(1_{G_F})\), and \(F_p\subset G\) is the Frobenius class of \(p\).
\end{lemma}

From now on, unless otherwise indicated, let \(K_1\) and \(K_2\) be number fields contained in a finite Galois extension \(K/\mathbb{Q}\) with group \(G\), and let \(G_F\) be the stabilizer in \(G\) of a subfield \(F\) of \(K\).  If \(K_1\) and \(K_2\) are \(s\)-equivalent for some splitting statistic \(s\), then the equality \(s(p,K_1)=s(p,K_2)\) at cofinitely many \(p\) implies equality over all primes unramified in \(K\).  Chebotarev density allows us to parlay this observation to relate \(G_{K_1}\) and \(G_{K_2}\) when \(s\) is a splitting statistic governing one of notions \ref{q}-\ref{uc} above.  The relations here are not new, but the approach is.  We record them for the sake of completeness, starting with a classic.
\begin{proposition}\label{qeqg}
If \(K_1\) and \(K_2\) are arithmetically equivalent, then \(G_{K_1}\) and \(G_{K_2}\) are almost conjugate in \(G\).
\end{proposition}
\begin{proof}
See Lemma 1 in \cite{perlis_equation_1977}.  
\end{proof}

\begin{proposition}\label{keqg}
If \(K_1\) and \(K_2\) are Kronecker equivalent, then \(G_{K_1}\) and \(G_{K_2}\) conjugacy cover one another in \(G\).
\end{proposition}
\begin{proof}
For every \(g\in G\), there is a rational prime \(p\) unramified in \(K\) whose Frobenius class \(F_p\) is precisely the conjugacy class of \(g\) in \(G\), so that \(1\in S_{F}(p)\) if and only if \(\chi_F(g)\neq 0\), by Lemma \ref{charactersandresidues}.  Since \(K_1\) and \(K_2\) are Kronecker equivalent, we conclude that \(\chi_{K_1}(g)\neq 0\) if and only if \(\chi_{K_2}(g)\neq 0\), so \(G_{K_1}\) and \(G_{K_2}\) conjugacy cover one another in \(G\).
\end{proof}

\begin{proposition}\label{wkeqg}
If \(K_1\) and \(K_2\) are weakly Kronecker equivalent, then the Sylow subgroups of \(G_{K_1}\) and \(G_{K_2}\) conjugacy cover one another in \(G\).
\end{proposition}
\begin{proof}
Let \(\delta_{i,p}=\gcd(S_{K_i}(p))\) and \(\chi_i=\chi_{K_i}\). If \(p\) is a rational prime unramified in \(K\), then \(\delta_{1,p}=\delta_{2,p}\).  Furthermore, letting \(o(g^G):=o(g)\), if \(o(F_p)\) is a \(q\)-power for some prime \(q\), then \(\delta_{i,p}\) is itself a residue over \(p\), so that \(\chi_i(F_p^n)\neq 0\) if and only if \(\delta_{i,p}|n\), by Lemma \ref{charactersandresidues}.  In particular, \(\chi_1(F_p)\neq 0\) if and only if \(\chi_2(F_p)\neq 0\).  Hence, by Chebotarev density, if \(g\) is an element of a Sylow \(q\)-subgroup of \(G\), then \(G_{K_1}\) intersects the \(G\)-conjugacy class of \(g\) if and only if \(G_{K_2}\) does, so the Sylow \(q\)-subgroups of \(G_{K_1}\) and \(G_{K_2}\) conjugacy cover one another in \(G\).  
\end{proof}

\begin{proposition}\label{guc}
If \(K_1\) and \(K_2\) are ultra-coarsely arithmetically equivalent, then \(G_{K_1}\) and \(G_{K_2}\) have the same normal core in \(G\).
\end{proposition}
\begin{proof}
By Proposition 1.5 in \cite{phagan_corresponding_2024}, we know that \(K_1\) and \(K_2\) have the same Galois closure over the rationals.  Basic Galois theory supplies the desired equality $$\bigcap\limits_{g\in G} gG_{K_1}g^{-1}=\bigcap\limits_{g\in G} gG_{K_2}g^{-1}.$$
\end{proof}

\section{Empty Exceptional Sets}\label{empty}
Keeping the notation from the previous section, we give new proofs that the exceptional set is empty for notions \ref{q}-\ref{wk}.

	\begin{theorem}\label{qeq}
		If \(K_1\) and \(K_2\) are arithmetically equivalent, then \(S_{K_1}(p)=S_{K_2}(p)\) for any rational prime \(p\).
	\end{theorem}
	\begin{proof}
		Let \(F\) be a subfield of \(K\).  Since \(\mathcal{I}\) is a normal subgroup of \(\mathcal{D}\), and \(\mathcal{D}=\mathcal{I}C\), there is a natural action of \(C\) on the double coset space \(\mathcal{I}\backslash G/G_F\), corresponding to the projection  \(\mathcal{I}\backslash G/G_F\rightarrow\mathcal{D}\backslash G/G_F\). Noting that \(c^m\mathcal{I}gG_F=\mathcal{I}gG_F\) if and only if \(g^{-1}\mathcal{I}c^mg\cap G_F\neq\emptyset\), a routine calculation verifies that we may identify \(S_F(p)\) with the multiset \(\mathcal{S}_{F,p}\) of cardinalities of orbits of this action.  Let \(S_m=C^m\mathcal{I}\backslash G\), and let \(S_{m,i}\) be the orbit space \(S_m/G_{K_i}\).  Observe that we may also think of \(S_{m,i}\) as the orbit space of the action of \(C^m\) on \(\mathcal{I}\backslash G/G_{K_i}\).  We know from Proposition \ref{qeqg} that \(G_{K_1}\) and \(G_{K_2}\) are almost conjugate in \(G\).  Therefore, by Proposition 2.4 in \cite{leininger_length_2007} and Lemma 2 in \cite{stuart_new_1995}, we know that \(|S_{m,1}|=|S_{m,2}|\) for each divisor \(m\) of \(|C|\), whence \(\mathcal{S}_{K_1,p}=\mathcal{S}_{K_2,p}\), by Theorem \ref{formula}.  Therefore, \(S_{K_1}(p)=S_{K_2}(p)\) for any prime \(p\), as desired.
	\end{proof}

	\begin{theorem}\label{KrEq}
		If \(K_1\) and \(K_2\) are Kronecker equivalent, then \(1\in S_{K_1}(p)\) if and only if \(1\in S_{K_2}(p)\) for any rational prime \(p\).
	\end{theorem}
	\begin{proof}
		 Suppose \(m\) is an integer divisible by some element of \(S_{K_1}(p)\), so there is a place \(\nu\) of \(K_1\) such that \(c^m\in\mathcal{I}\mathcal{D}_{\omega/\nu}\), by Lemma \ref{charactersandresidues}.  In particular, there are \(s\in\mathcal{I}\) and \(t\in\mathcal{D}_{\omega/\nu}\) such that \(c^m = st\).  Now, \(t\in G_{K_1}\), so there is \(g\in G\) such that \(g tg^{-1}\in G_{K_2}\), by Proposition \ref{keqg}.  But then, if \(\omega'=g\omega\), we also know that \(g tg^{-1}\in\mathcal{D}_{\omega'/p}\), and \(g sg^{-1}\in\mathcal{I}_{\omega'/p}\). Hence \(g c^m g^{-1}\in \mathcal{I}_{\omega'/p}\mathcal{D}_{\omega'/\nu'}\), where \(\nu'\) is the place of \(K_2\) divisible by \(\omega'\).  Therefore, there is a divisor of \(m\) in \(S_{K_2}(p)\), again by Lemma \ref{charactersandresidues}.  By symmetry, we conclude that, for \(m\in\mathbb{Z}\), there is a divisor of \(m\) in \(S_{K_1}(p)\) if and only if there is a divisor of \(m\) in \(S_{K_2}(p)\).  For \(m=1\), this means \(1\in S_{K_1}(p)\) if and only if \(1\in S_{K_2}(p)\).
	\end{proof}
	
	\begin{theorem}
		If \(K_1\) and \(K_2\) are weakly Kronecker equivalent, then \(\gcd(S_{K_1}(p))=\gcd(S_{K_2}(p))\) for any rational prime \(p\).
	\end{theorem}
	\begin{proof}
	Given a prime \(q\), Lemma \ref{charactersandresidues} tells us that \(\delta_{i,p}\) is not divisible by \(q^k\) if and only if there is a place \(\omega\) of \(K\) such that \(c^{n_qq^{k-1}}\in\mathcal{I}(\mathcal{D}\cap G_{K_i})\), where \(n_q=o(c)/q^{v_q(o(c))}\).  Indeed, given a divisor \(f\) of \(o(c)\), we know that \(f|n_qq^{k-1}\) if and only if \(q^k\) does not divide \(f\).  Now, \(c^{n_q}\) has \(q\)-power order, so if \(q^k\) does not divide \(\delta_{1,p}\), by Lemma \ref{SylowPullBack}, we may write \(c^{n_qq^{k-1}}=\iota\gamma\) for some \(\iota\in\mathcal{I}\) and \(\gamma\) an element of a Sylow \(q\)-subgroup of \(\mathcal{D}_{\omega/\nu}\), where \(\nu\) is the place of \(K_1\) under \(\omega\).  We can then use Proposition \ref{wkeqg}, along with the trick from the proof of Theorem \ref{KrEq} to deduce that \(q^k|\delta_{1,p}\) if and only if \(q^k|\delta_{2,p}\).  Hence, \(\delta_{1,p}\) and \(\delta_{2,p}\) have the same prime divisors, with the same multiplicity, so \(\delta_{1,p}=\delta_{2,p}\).
	\end{proof}

\begin{remark}
Our argument has shown that the group-theoretic \textit{implications} we derived in Section \ref{galois} for notions \ref{q}-\ref{wk} are in fact equivalent \textit{characterizations} of their associated notion of arithmetic similarity. Such characterizations have been another recurring theme in the literature.
\end{remark}
	
\section{Non-Empty Exceptional Sets}\label{nempty}
Since the historically studied notions of arithmetic similarity are defined by splitting statistics with empty exceptional sets, it is natural to wonder if there are notions admitting non-empty exceptional sets.  We give two examples of such notions.
\subsection*{Example 1}
It was shown in \cite{phagan_corresponding_2024} that number fields are ultra-coarsely arithmetically equivalent if and only if they have the same Galois closure over \(\mathbb{Q}\).  Therefore, to have a non-empty exceptional set in this case, it suffices to a find a prime \(p\) and a number field \(F\) with Galois closure \(K\) such that \(\lcm(S_F(p))\neq\lcm(S_K(p))\).  Proposition \ref{guc}, along with a bit of Galois theory, reveals that if one can arrange so that $$\mathcal{I}\neq\cap_{g\in G}\mathcal{I}(\mathcal{D}\cap gG_Fg^{-1}),$$ then \(\lcm(S_F(p))\neq\lcm(S_K(p))\).  This is an entirely reasonable condition, and indeed what is likely the first example one would consider produces a non-empty exceptional set.  
\begin{theorem}\label{nemptythm}
There are ultra-coarsely arithmetically equivalent fields with a non-empty exceptional set.
\end{theorem} 
\begin{proof}
Let \(F_1=\mathbb{Q}(\sqrt[3]{2})\), \(F_2=\mathbb{Q}(\sqrt{-3})\), and \(K=F_1F_2\).  Then \(K\) is the Galois closure of \(F_1\) over \(\mathbb{Q}\).  Using monogenicity of \(F_1\) and \(F_2\), along with the Dedekind-Kummer Theorem, one can calculate directly that \(S_{F_1}(2)=\{1\}\), \(S_{F_2}(2)=\{2\}\), \(S_{K}(2)=\{2\}\).  In particular, \(\lcm(S_K(2))\neq\lcm(S_{F_1}(2))\).
\end{proof}

\subsection*{Example 2}
Let \(s\) be the splitting statistic 
$$
s(p,K)=\begin{cases}1, & S_K(p)=\{1\}\\ 0, &\text{otherwise}\end{cases}
$$
Observe that the induced notion of arithmetic similarity consists of two sets: the singleton \(\{\mathbb{Q}\}\) and its complement \(\mathcal{N}-\{\mathbb{Q}\}\).  Furthermore, if \(K\neq \mathbb{Q}\), then \(s(p,K)=1\) if and only if \(p\) is totally ramified in \(K\).  Therefore, if \(K_1\) and \(K_2\) are number fields not equal to \(\mathbb{Q}\), and there is a rational prime \(p\) totally ramified in \(K_1\) but not in \(K_2\) (e.g. \(K_1=\mathbb{Q}(i), K_2=\mathbb{Q}(i,\zeta_3)\), with \(\zeta_3\) a primitive third root of unity), then the exceptional set of \(s\) for \(K_1\) and \(K_2\) is non-empty.  We record this as a theorem.
\begin{theorem}
The exceptional set of \(s\) as above is non-empty for some number fields.
\end{theorem}

\begin{remark}
One may associate to a number field \(F\) and splitting statistic \(s\) a Dirichlet series \(D_{F,s}\) given by a polynomial Euler product
$$
D_{F,s}(z)=\prod\limits_{p\in\mathbb{P}}\prod\limits_{m\in\mathbb{Z}}(1-p^{-mz})^{-s(p,F)(m)}
$$
The examples in this section, along with the Strong Multiplicity One Theorem in \cite{kaczorowski_strong_2001}, indicate that \(D_{F,s}\) is not always of Selberg class.  On the other hand, if \(s(p,F)=S_F(p)\), then \(D_{F,s}=\zeta_F\) is a classic example of a Selberg class Dirichlet series.  The failure of a given \(D_{F,s}\) to be Selberg can be somewhat subtle.  For example, if \(K/\mathbb{Q}\) is Galois, and
$$
s(p,F)(n)=\begin{cases}
|S_K(p)|, & n=\lcm(S_F(p)) \\
0, & \text{otherwise}
\end{cases}
$$
then \(D_{K,s}=\zeta_K\), and for \(F\) a number field with Galois closure \(K\), we find that \(D_{F,s}\) differs from \(\zeta_K\) at only finitely many Euler factors, by Proposition 1.5 in \cite{phagan_corresponding_2024} (note also that the notion of arithmetic similarity induced by \(s\) is ultra coarse arithmetic equivalence). 
\end{remark}

\section{Final Remarks}\label{parting}
Notions of arithmetic similarity have thus far been introduced on an ad-hoc basis: \ref{q} was introduced by Gassmann as a counterpoint to Kronecker's program of determining a number field up to isomorphism with the decomposition laws of rational primes, \ref{k} was introduced as a means of validating Kronecker's program for Bauerian fields, \ref{wk} was introduced as a natural generalization of \ref{k}, and \ref{uc} was introduced to understand coarse relations between Dedekind zeta functions of some corresponding abelian extensions.  A more systematic approach might prove fruitful.  An obvious question of interest is exactly which splitting statistics always have empty exceptional sets, or in other words, which abide by a Strong Multiplicity One Theorem.  We conclude by mentioning two more.  
\begin{enumerate}
\item The splitting statistic inducing a given notion of arithmetic similarity is not necessarily unique.  For instance, arithmetic equivalence is induced by either of the statistics \((p,F)\mapsto S_F(p)\) or \((p,F)\mapsto |S_F(p)|\), and Kronecker equivalence is induced by either 
$$
(p,F)\mapsto\begin{cases}
1 & \text{if }1\in S_F(p),\\
0 & \text{ otherwise}
\end{cases}
$$
or \((p,F)\mapsto \min S_F(p)\). It would be interesting to know whether two statistics inducing the same notion have equal exceptional sets.  This seems likely and is indeed the case in every example known to the author.
\item There is a parallel topic in Riemannian geometry, which we will refer to as ``spectral similarity." The interested reader might first consult \cite{sunada_riemannian_1985} for a classic.  Here the equivalence relations of interest reflect properties of manifolds' length or eigenvalue spectrum.  Certain notions of spectral similarity have an obvious arithmetic counterpart, due to the coincidence of the arithmetic and covering Galois theories.  For instance, \textit{length equivalence}, defined in \cite{leininger_length_2007} as a generalization of length isospectrality, is the spectral analog of Kronecker equivalence.  The notion of \textit{eigenvalue equivalence} is also studied there as a generalization of isospectrality.  To the author's knowledge, the arithmetic analog of eigenvalue equivalence has not been studied, and it is unclear whether or not the resulting partition of \(\mathcal{N}\) would be a notion of arithmetic similarity.  
\end{enumerate}
	
\bibliography{ExceptionsInArithmeticSimilarity}
\bibliographystyle{alpha}
\end{document}